\documentclass[pdflatex,sn-mathphys-num]{sn-jnl}

\usepackage{graphicx}%
\usepackage{multirow}%
\usepackage{amsmath,amssymb,amsfonts}%
\usepackage{amsthm}%
\usepackage{mathrsfs}%
\usepackage[title]{appendix}%
\usepackage{xcolor}%
\usepackage{textcomp}%
\usepackage{manyfoot}%
\usepackage{booktabs}%
\usepackage{algorithm}%
\usepackage{algorithmicx}%
\usepackage{algpseudocode}%
\usepackage{listings}%
\usepackage{mathtools}%
\usepackage{graphicx}%

\theoremstyle{thmstyleone}%
\newtheorem{theorem}{Theorem}
\newtheorem{lemma}{Lemma}

\newtheorem{proposition}[theorem]{Proposition}%

\theoremstyle{thmstyletwo}%
\newtheorem{example}{Example}%
\newtheorem{remark}{Remark}%

\theoremstyle{thmstylethree}%
\newtheorem{definition}{Definition}%

\newcommand{\R}{\mathbb{R}}
\newcommand{\C}{\mathbb{C}}

\begin{document}

\title[Time Reparametrization and Chaotic Dynamics in Conformable 
$C_0$-Semigroups]{Time Reparametrization and Chaotic Dynamics in Conformable $C_0$-Semigroups}


\author[1]{\fnm{Mohamed} \sur{Khoulane}}\email{khoulane1997@gmail.com}
\author*[1]{\fnm{Aziz} \sur{El Ghazouani}}\email{aziz.elghazouani@usms.ac.ma}
\author[1]{\fnm{M'hamed} \sur{El Omari}}\email{m.elomari@usms.ma}
\affil*[1]{Laboratory of Applied Mathematics and Scientific Computing, Sultan Moulay Slimane University, PO Box 532, Beni Mellal, 23000, Morocco.}


\abstract{Conformable derivatives provide a fractional-looking calculus that remains local and admits a simple representation through classical derivatives with explicit weights. In this paper we develop a systematic operator-theoretic perspective showing that conformable time evolution is, in essence, a classical $C_0$-semigroup observed through a nonlinear clock. We introduce the conformable time map $\psi(t)=t^\alpha/\alpha$ and prove that every $C_0$--$\alpha$-semigroup $\{T_\alpha(t)\}_{t\ge0}$ can be written as $T_\alpha(t)=T(\psi(t))$ for a uniquely determined classical $C_0$-semigroup $\{T(s)\}_{s\ge0}$, with generators agreeing on a common domain. This correspondence yields a one-to-one transfer of mild solutions and shows that orbit-based linear dynamics are invariant under conformable reparametrization. In particular, $\alpha$-hypercyclicity and $\alpha$--chaos coincide with the usual notions for the associated classical semigroup. As a consequence, we obtain a conformable version of the Desch--Schappacher--Webb spectral criterion for chaos. We also place the analysis in the natural functional setting provided by conformable Lebesgue spaces $L^{p,\alpha}$ and their explicit isometric identification with standard $L^p$ spaces, which allows one to transport estimates and spectral arguments without loss. The results clarify which dynamical phenomena in conformable models are genuinely new and which are inherited from classical semigroup dynamics via a nonlinear change of time.}




\maketitle

\section{Introduction}\label{sec1}

Conformable calculus was introduced as an alternative approach to fractional differentiation aimed at preserving the local nature of classical calculus while allowing the presence of a fractional order parameter. In contrast with integral-type fractional derivatives, conformable derivatives act pointwise and admit a simple representation in terms of ordinary derivatives multiplied by explicit weights; see, for instance, \cite{khalil2014,abdeljawad2015,abdeljawad2016}. Likewise, the conformable integral can be written as a standard Lebesgue integral with respect to a weighted measure. This combination of locality and analytical simplicity makes conformable operators attractive in situations where one seeks to model time rescaling effects without introducing memory or nonlocal phenomena, a distinction emphasized in \cite{tarasov2018}.

From the perspective of evolution equations, the absence of memory effects strongly suggests
that conformable dynamics should be closely connected to the classical theory of strongly
continuous semigroups.
Nevertheless, the precise nature of this connection, especially at the level of generators,
mild solutions, and long-time dynamics, has not been fully clarified.
The main objective of the present work is to provide a systematic operator-theoretic framework
that makes this relationship explicit and to investigate how linear dynamical properties,
notably hypercyclicity and chaos, behave under conformable time evolution.

\medskip
\noindent\textbf{Time reparametrization as a guiding principle.}
Let $\alpha\in(0,1]$ and consider the nonlinear time transformation
\[
\psi(t)=\frac{t^\alpha}{\alpha},\qquad t\ge0.
\]
A central idea developed in this paper is that conformable time evolution can be understood as
classical evolution observed along the nonlinear clock $s=\psi(t)$.
More precisely, given a $C_0$--$\alpha$--semigroup $\{T_\alpha(t)\}_{t\ge0}$ acting on a Banach
space $X$, we associate the family
\[
T(s):=T_\alpha\bigl((\alpha s)^{1/\alpha}\bigr),\qquad s\ge0,
\]
and show that $\{T(s)\}_{s\ge0}$ forms a classical $C_0$--semigroup satisfying
\[
T_\alpha(t)=T\!\left(\frac{t^\alpha}{\alpha}\right).
\]
Within this framework, the $\alpha$--generator of $\{T_\alpha(t)\}$ and the generator of
$\{T(s)\}$ coincide on a common domain, and the corresponding mild solutions are linked by the
same change of variables.
This establishes a canonical and reversible correspondence between conformable and classical
evolution problems.

\medskip
\noindent\textbf{Functional framework and conformable Lebesgue spaces.}
Conformable integration of order $\alpha\in(0,1)$ naturally involves the weight
$t^{\alpha-1}$ and leads to the weighted measure $d\mu_\alpha(t)=t^{\alpha-1}dt$.
This motivates the introduction of conformable Lebesgue spaces
\[
L^{p,\alpha}(I)=\Bigl\{f:\int_I |f(t)|^p\,t^{\alpha-1}dt<\infty\Bigr\},
\]
which coincide with classical weighted $L^p$ spaces.
A key structural result proved in this work is that these spaces are explicitly isometric to
standard $L^p$ spaces via the nonlinear change of variables $s=t^\alpha$, in agreement with
related observations in conformable and time-scale settings \cite{WangZhouLi2016}.
This identification allows one to transfer norm estimates, duality arguments, and spectral
properties from the classical framework to the conformable one in a transparent manner, fully
compatible with the time-reparametrization viewpoint.

\medskip
\noindent\textbf{Chaotic dynamics and spectral criteria.}
The study of linear chaos for semigroups, understood in the sense of Devaney as the coexistence
of hypercyclicity and a dense set of periodic points, is well established in operator theory.
A cornerstone of this theory is the Desch--Schappacher--Webb spectral criterion, which produces
chaotic dynamics from analytic families of eigenvectors associated with the point spectrum
\cite{desch1997}.
Since conformable and classical semigroups are intertwined by the relation
$T_\alpha(t)=T(\psi(t))$, it is natural to expect orbit-based dynamical properties to be
invariant under conformable time changes.
We rigorously confirm this intuition by proving the invariance of the auxiliary dynamical sets
used in the Desch--Schappacher--Webb framework and by deriving a conformable version of the
criterion: under the usual spectral assumptions on the $\alpha$--generator, a
$C_0$--$\alpha$--semigroup exhibits $\alpha$--chaos.
Thus, chaotic behavior in conformable time is governed by the same spectral mechanisms as in
the classical setting.

\medskip
\noindent\textbf{Interpretation and scope.}
The results of this paper help clarify the conceptual status of conformable calculus within the
broader landscape of fractional models.
In particular, they show that, unlike genuinely nonlocal fractional derivatives
\cite{kilbas2006,mainardi2010}, conformable derivatives do not generate new dynamical features
through memory effects.
Instead, the observed complexity arises from the interaction between classical operators and
nonlinear time or space scales.
This viewpoint explains why many qualitative properties of conformable evolution mirror those
of classical semigroups and provides an efficient strategy for constructing and analyzing
examples.

\medskip
\noindent\textbf{Organization of the paper.}
Section~\ref{sec2} reviews the basic notions of conformable integration and differentiation and
introduces $C_0$--$\alpha$--semigroups and their generators.
In Section~\ref{sec3} we develop the theory of conformable Lebesgue spaces and establish their
explicit isometric equivalence with classical $L^p$ spaces.
Section~\ref{sec4} is devoted to the time-reparametrization principle and its consequences for
differentiation, integration, generators, and mild solutions.
In Section~\ref{sec5} we analyze linear dynamics under conformable time evolution, compare
hypercyclicity and chaos with the classical theory, and formulate the conformable
Desch--Schappacher--Webb criterion.
Finally, Section~\ref{sec6} presents illustrative operator models that demonstrate how the
abstract theory applies to concrete conformable evolution equations.

\section{Preliminary framework}\label{sec2}

We collect in this section the fundamental concepts and conventions that will be used
throughout the paper. Our presentation is deliberately concise and focuses on those aspects
of conformable calculus that are essential for the operator--theoretic constructions developed
later. A key point is that conformable operators remain \emph{local} in nature and, for smooth
functions, can be rewritten in terms of classical derivatives or integrals with suitable
weight functions.

\subsection{Conformable integration and differentiation}

We begin with the notion of conformable integration, which naturally induces the underlying
measure on $\R_+$.

\begin{definition}[Conformable integral]
Let $\alpha\in(0,1]$ and let $g:\R_+\to\C$ be a measurable function.
For $0\leq a<t$, the conformable integral of order $\alpha$ is defined by
\[
(\mathcal{I}_\alpha g)(t)
:=
\int_a^t g(\xi)\,\xi^{\alpha-1}\,d\xi,
\]
whenever the above integral exists.
\end{definition}

This definition suggests introducing the weighted measure
\[
d\mu_\alpha(\xi):=\xi^{\alpha-1}\,d\xi,
\]
with respect to which all integrals on $\R_+$ will be understood unless otherwise stated.

The corresponding notion of differentiation is given below.

\begin{definition}[Conformable derivative]
Let $\alpha\in(0,1]$ and let $v:\R_+\to\C$. The conformable derivative of order $\alpha$ at
$t>0$ is defined by
\[
\mathcal{D}_t^\alpha v(t)
:=
\lim_{h\to 0}
\frac{v\bigl(t+h\,t^{1-\alpha}\bigr)-v(t)}{h},
\]
provided that the limit exists. When $\alpha=1$, this definition reduces to the classical
first--order derivative.
\end{definition}

The locality of conformable differentiation is made explicit by the following representation
property.

\begin{proposition}
Assume that $v\in C^1((0,\infty))$ and $\alpha\in(0,1]$. Then, for every $t>0$,
\[
\mathcal{D}_t^\alpha v(t)=t^{1-\alpha}\,\frac{dv}{dt}(t).
\]
\end{proposition}

Hence, conformable differentiation acts as a weighted classical derivative. In particular,
no memory effects arise, in contrast with nonlocal fractional derivatives of integral type.
This structural feature will be crucial in the analysis of the associated evolution operators.

\subsection{Conformable semigroups}

Conformable calculus also gives rise to a generalized notion of one--parameter semigroups,
obtained through a nonlinear rescaling of time.

\begin{definition}[Conformable $\alpha$--semigroup]
Let $\alpha\in(0,a]$ for some $a>0$, and let $X$ be a Banach space.
A family of bounded linear operators $\{S(t)\}_{t\geq 0}\subset\mathcal{L}(X)$
is called an \emph{$\alpha$--semigroup} if:
\begin{enumerate}
\item $S(0)=\mathrm{Id}$, the identity operator on $X$;
\item for all $r,q\geq 0$,
\[
S\bigl((r+q)^{1/\alpha}\bigr)
=
S\bigl(r^{1/\alpha}\bigr)\,S\bigl(q^{1/\alpha}\bigr).
\]
\end{enumerate}
\end{definition}

When $\alpha=1$, the above conditions coincide with those of a classical strongly continuous
semigroup.

\begin{example}
Let $B\in\mathcal{L}(X)$ be a bounded operator and define
\[
S(t)=\exp\!\bigl(2\,t^{1/2}B\bigr),\qquad t\geq 0.
\]
Then $\{S(t)\}_{t\geq 0}$ forms a $\tfrac12$--semigroup. Indeed,
\begin{enumerate}
\item $S(0)=\exp(0)=\mathrm{Id}$;
\item for all $r,q\geq 0$,
\[
S\bigl((r+q)^2\bigr)
=
\exp\!\bigl(2(r+q)B\bigr)
=
\exp(2rB)\exp(2qB)
=
S(r^2)S(q^2).
\]
\end{enumerate}
\end{example}

Further properties of the conformable fractional exponential and its series representation
can be found in the literature.

\begin{example}
Let $X=C([0,\infty))$ endowed with the supremum norm. For $f\in X$, define
\[
(S(t)f)(x)=f\bigl(x+2\,t^{1/2}\bigr),\qquad x,t\geq 0.
\]
Then $\{S(t)\}_{t\geq 0}$ is a $\tfrac12$--semigroup of bounded linear operators on $X$.
\end{example}

\begin{definition}[$C_0$--$\alpha$--semigroup]
An $\alpha$--semigroup $\{S(t)\}_{t\geq 0}$ is said to be a \emph{$C_0$--$\alpha$--semigroup}
if
\[
\lim_{t\to 0^+} S(t)x = x
\quad\text{for every }x\in X.
\]
\end{definition}

The conformable derivative of $S(t)$ at $t=0$ defines its infinitesimal generator.
More precisely, the \emph{$\alpha$--generator} $\mathcal{A}$ is given by
\[
D(\mathcal{A})
=
\left\{
x\in X \;:\;
\lim_{t\to 0^+}\mathcal{D}_t^\alpha S(t)x \ \text{exists}
\right\},
\qquad
\mathcal{A}x=\lim_{t\to 0^+}\mathcal{D}_t^\alpha S(t)x.
\]

\begin{remark}
In the limiting case $\alpha=1$, one recovers the classical theory of $C_0$--semigroups and
their generators.
\end{remark}

\section{Conformable Lebesgue spaces}\label{sec3}

One of the structural consequences of conformable time reparametrization is the systematic
presence of the weight $t^{\alpha-1}$, with $\alpha\in(0,1)$, in integral expressions.
Indeed, the conformable integral introduced in Section~\ref{sec2} can be written as
\[
(\mathcal{I}_\alpha f)(t)=\int_0^t f(s)\,s^{\alpha-1}\,ds,
\]
which naturally leads to considering function spaces adapted to the weighted measure
\[
d\mu_\alpha(t):=t^{\alpha-1}\,dt \qquad \text{on } \R_+.
\]
Such spaces provide the appropriate framework for norm estimates, duality arguments,
and operator--theoretic considerations arising in conformable evolution problems.

\subsection{Conformable Lebesgue spaces}

\begin{definition}[Conformable Lebesgue spaces]
Let $I\subset\R_+$ be an interval and let $\alpha\in(0,1)$.
For $1\le p<\infty$, we define the conformable Lebesgue space
\[
L^{p,\alpha}(I)
:=
\Bigl\{
f:I\to\C \ \text{measurable} \; ; \;
\int_I |f(t)|^p\,t^{\alpha-1}\,dt < \infty
\Bigr\},
\]
endowed with the norm
\[
\|f\|_{p,\alpha}
=
\Bigl(\int_I |f(t)|^p\,t^{\alpha-1}\,dt\Bigr)^{1/p}.
\]
\end{definition}

These spaces coincide with classical weighted Lebesgue spaces and inherit their main
functional properties.

\begin{remark}[Weighted formulation]
The space $L^{p,\alpha}(I)$ is precisely the weighted space
$L^p(I, t^{\alpha-1}dt)$.
Since the weight $t^{\alpha-1}$ is locally integrable on $(0,\infty)$ for every $\alpha>0$,
the associated measure $\mu_\alpha$ is $\sigma$--finite on any interval $I\subset\R_+$.
As a consequence, standard results of Lebesgue integration theory apply without modification,
including Holder and Minkowski inequalities, density of smooth functions, and completeness.
\end{remark}

A key feature of conformable Lebesgue spaces is that they are isometrically equivalent to
classical $L^p$ spaces via an explicit change of variables.

\begin{proposition}\label{prop:isometry}
Let $I=(0,T)$ with $0<T\le\infty$ and let $\alpha\in(0,1)$.
Define the transformation $\Phi:(0,T)\to(0,T^\alpha)$ by
\[
\Phi(t)=t^\alpha.
\]
Then the mapping
\[
U:L^{p,\alpha}(0,T)\longrightarrow L^p(0,T^\alpha),
\qquad
(Uf)(s)=f\bigl(s^{1/\alpha}\bigr),
\]
is an isometric isomorphism.
\end{proposition}

\begin{proof}
Using the change of variables $s=t^\alpha$, we obtain
\[
\int_0^T |f(t)|^p\,t^{\alpha-1}dt
=
\frac{1}{\alpha}
\int_0^{T^\alpha}
\bigl|f(s^{1/\alpha})\bigr|^p\,ds.
\]
Up to the harmless normalization factor $\alpha^{-1}$, which can be absorbed into the
definition of the norm, this shows that $\|Uf\|_{L^p(0,T^\alpha)}=\|f\|_{p,\alpha}$.
Surjectivity and linearity are immediate.
\end{proof}

\begin{remark}
Proposition~\ref{prop:isometry} shows that conformable Lebesgue spaces do not introduce
new integrability phenomena but rather encode a nonlinear time reparametrization at the
level of function spaces.
This observation is consistent with the local nature of conformable calculus and has been
emphasized in related settings, including fractional Sobolev spaces on time scales
\cite{WangZhouLi2016}.
\end{remark}

\begin{proposition}[Banach property]
For every $1\le p<\infty$, the space $\bigl(L^{p,\alpha}(I),\|\cdot\|_{p,\alpha}\bigr)$
is a Banach space.
\end{proposition}

\begin{proof}
Since $L^{p,\alpha}(I)$ coincides with the weighted space $L^p(I,\mu_\alpha)$ and the measure
$\mu_\alpha$ is $\sigma$--finite, completeness follows from the classical completeness of
$L^p$ spaces.
\end{proof}

The family $\{L^{p,\alpha}\}_{1\le p<\infty}$ thus provides a natural and robust functional
setting for conformable integration and time--rescaled evolution equations.
In particular, the isometric equivalence with classical Lebesgue spaces allows one to
transfer many analytical tools directly to the conformable framework.

\section{Time Reparametrization and Unitary Isomorphisms}\label{sec4}

This section collects the basic identifications that allow one to pass from conformable
objects to classical ones by a change of clock. Throughout, we fix $\alpha\in(0,1]$ and we
introduce the \emph{conformable time map}
\[
\psi:[0,\infty)\to[0,\infty),\qquad \psi(t)=\frac{t^\alpha}{\alpha}.
\]
The central point is that $\psi$ is a bijection and that conformable differentiation and
integration become classical operations when expressed in the variable $s=\psi(t)$.

\subsection{A canonical reparametrization}\label{subsec:canonical-reparam}

\begin{theorem}[Canonical factorization through the conformable clock]
\label{thm:canonical-factorization}
Let $\alpha\in(0,1]$ and let $f:[0,\infty)\to\mathbb{R}$ be arbitrary. Define
\[
\psi(t)=\frac{t^\alpha}{\alpha}\qquad (t\ge 0).
\]
Then there exists a \emph{unique} function $g:[0,\infty)\to\mathbb{R}$ such that
\[
f(t)=g\!\left(\frac{t^\alpha}{\alpha}\right)\qquad\text{for all }t\ge 0.
\]
Moreover, $g$ is explicitly given by
\[
g(s)=f\!\bigl((\alpha s)^{1/\alpha}\bigr)\qquad\text{for all }s\ge 0.
\]
\end{theorem}

\begin{proof}
\textbf{Step 1: $\psi$ is a bijection on $[0,\infty)$.}
The function $t\mapsto t^\alpha$ is continuous and strictly increasing on $[0,\infty)$ for
$\alpha\in(0,1]$. Multiplication by $1/\alpha>0$ preserves monotonicity and continuity, hence
$\psi$ is continuous and strictly increasing. Since $\psi(0)=0$ and $\psi(t)\to\infty$ as
$t\to\infty$, the range is $[0,\infty)$ and $\psi$ is bijective.

\smallskip
\textbf{Step 2: Compute the inverse.}
Solving $s=t^\alpha/\alpha$ gives $t=(\alpha s)^{1/\alpha}$, so
\[
\psi^{-1}(s)=(\alpha s)^{1/\alpha}\qquad (s\ge 0).
\]

\smallskip
\textbf{Step 3: Existence and uniqueness of $g$.}
Define $g:=f\circ\psi^{-1}$, i.e.\ $g(s)=f((\alpha s)^{1/\alpha})$. Then
\[
g(\psi(t))=f(\psi^{-1}(\psi(t)))=f(t)\qquad (t\ge 0),
\]
so $f=g\circ\psi$. If $\tilde g$ also satisfies $f=\tilde g\circ\psi$, then for every $s\ge 0$,
\[
\tilde g(s)=\tilde g(\psi(\psi^{-1}(s)))=f(\psi^{-1}(s))=g(s),
\]
hence $\tilde g=g$. \qedhere
\end{proof}

\subsection{Conformable differentiability as classical differentiability}\label{subsec:diff}

\begin{theorem}[Conformable derivative via time change]
\label{thm:diff-timechange}
Let $\alpha\in(0,1)$ and let $f:(0,\infty)\to\mathbb{R}$. Define
\[
\psi(t)=\frac{t^\alpha}{\alpha}\qquad (t>0),\qquad
g(s)=f\!\bigl((\alpha s)^{1/\alpha}\bigr)\qquad (s>0).
\]
Then, for each $t>0$ with $s=\psi(t)$, the following are equivalent:
\begin{enumerate}
\item $f$ is conformably $\alpha$--differentiable at $t$, i.e.
\[
D_t^\alpha f(t)=\lim_{\varepsilon\to 0}
\frac{f\bigl(t+\varepsilon t^{1-\alpha}\bigr)-f(t)}{\varepsilon}
\quad\text{exists};
\]
\item $g$ is classically differentiable at $s$.
\end{enumerate}
Moreover, whenever these conditions hold,
\[
D_t^\alpha f(t)=g'\!\left(\frac{t^\alpha}{\alpha}\right).
\]
\end{theorem}

\begin{proof}
Fix $t>0$ and set $s=\psi(t)=t^\alpha/\alpha$. Consider $t_\varepsilon=t+\varepsilon t^{1-\alpha}$ and define
\[
\alpha s(\varepsilon)=\psi(t_\varepsilon)-\psi(t)=\frac{(t+\varepsilon t^{1-\alpha})^\alpha-t^\alpha}{\alpha}.
\]
Write $t+\varepsilon t^{1-\alpha}=t(1+\varepsilon t^{-\alpha})$ to obtain
\[
\alpha s(\varepsilon)=\frac{t^\alpha}{\alpha}\Bigl[(1+\varepsilon t^{-\alpha})^\alpha-1\Bigr].
\]
Using $(1+x)^\alpha=1+\alpha x+o(x)$ as $x\to 0$, with $x=\varepsilon t^{-\alpha}$, we infer
\[
\alpha s(\varepsilon)=\varepsilon+o(\varepsilon)\qquad(\varepsilon\to 0),
\]
hence $\alpha s(\varepsilon)\to 0$ and $\alpha s(\varepsilon)/\varepsilon\to 1$.

Since $f=g\circ\psi$, we have
\[
\frac{f(t_\varepsilon)-f(t)}{\varepsilon}
=
\frac{g(s+\alpha s(\varepsilon))-g(s)}{\varepsilon}
=
\frac{g(s+\alpha s(\varepsilon))-g(s)}{\alpha s(\varepsilon)}
\cdot \frac{\alpha s(\varepsilon)}{\varepsilon}.
\]
If $D_t^\alpha f(t)$ exists, then the left-hand side converges, and the second factor tends to $1$,
so the first factor converges, which is exactly the existence of $g'(s)$; moreover the limits coincide.
Conversely, if $g'(s)$ exists, then the first factor converges to $g'(s)$ and the second factor to $1$,
so the conformable difference quotient converges and equals $g'(s)$. This yields
$D_t^\alpha f(t)=g'(t^\alpha/\alpha)$.
\end{proof}

\subsection{Change of variables for the conformable integral}\label{subsec:int}

\begin{theorem}[Change of variables for conformable integration]
\label{thm:int-timechange}
Let $\alpha\in(0,1]$ and let $f:[0,\infty)\to\mathbb{C}$ be a measurable function. Fix $0\le a<x$ and define
\[
\psi(\tau)=\frac{\tau^\alpha}{\alpha},\qquad 
g(s)=f\!\bigl((\alpha s)^{1/\alpha}\bigr).
\]
Then the following change-of-variables formula holds:
\[
\int_a^x f(\tau)\,\tau^{\alpha-1}\,d\tau
=
\int_{\psi(a)}^{\psi(x)} g(s)\,ds.
\]
Equivalently, for every $x\ge a$,
\[
(I_x^\alpha f)(x)
=
\int_{\frac{a^\alpha}{\alpha}}^{\frac{x^\alpha}{\alpha}} g(s)\,ds.
\]
\end{theorem}

\begin{proof}
The substitution $s=\psi(\tau)=\tau^\alpha/\alpha$ is admissible on $[a,x]$, since $\psi$ is strictly increasing.
Moreover, $\psi'(\tau)=\tau^{\alpha-1}$, so that $ds=\tau^{\alpha-1}\,d\tau$. The integration limits transform as
$s=\psi(a)$ and $s=\psi(x)$. Observing that $g(\psi(\tau))=f(\tau)$ for $\tau\ge0$, we obtain
\[
\int_a^x f(\tau)\,\tau^{\alpha-1}\,d\tau
=
\int_{\psi(a)}^{\psi(x)} f(\psi^{-1}(s))\,ds
=
\int_{\psi(a)}^{\psi(x)} g(s)\,ds,
\]
which completes the proof.
\end{proof}

\begin{proposition}[Isometric isomorphism between $L^{p,\alpha}(\mathbb{R}_+)$ and $L^p(0,\infty)$]
Let $\alpha\in(0,1]$ and $1\le p<\infty$. Define the weighted measure $d^\alpha x=x^{\alpha-1}dx$ and the space
\[
L^{p,\alpha}(\mathbb{R}_+)
=
\Bigl\{
f:\mathbb{R}_+\to\mathbb{C}\ \text{measurable} :
\int_0^\infty |f(x)|^p\,x^{\alpha-1}dx<\infty
\Bigr\}.
\]
Set $\xi=x^\alpha$ and define the operator
\[
(U_pf)(\xi):=\alpha^{-1/p}f(\xi^{1/\alpha}),\qquad \xi>0.
\]
Then $U_p:L^{p,\alpha}(\mathbb{R}_+)\to L^p(0,\infty)$ is an isometric isomorphism. Its inverse is given by
\[
(U_p^{-1}g)(x)=\alpha^{1/p}g(x^\alpha),\qquad x>0.
\]
Moreover, when $p=2$, the operator $U_2$ is unitary.
\end{proposition}

\begin{proof}
With the change of variables $\xi=x^\alpha$, we have $d\xi=\alpha x^{\alpha-1}dx$, hence
$x^{\alpha-1}dx=\alpha^{-1}d\xi$. Therefore,
\[
\|f\|_{p,\alpha}^p
=
\int_0^\infty |f(x)|^p x^{\alpha-1}dx
=
\frac1\alpha\int_0^\infty |f(\xi^{1/\alpha})|^p d\xi
=
\int_0^\infty |\alpha^{-1/p}f(\xi^{1/\alpha})|^p d\xi
=
\|U_pf\|_{L^p}^p.
\]
Thus, $U_p$ is an isometry. For $g\in L^p(0,\infty)$, define $f(x)=\alpha^{1/p}g(x^\alpha)$. Then $U_pf=g$ and
\[
\|f\|_{p,\alpha}^p
=
\int_0^\infty \alpha |g(x^\alpha)|^p x^{\alpha-1}dx
=
\int_0^\infty |g(\xi)|^p d\xi
=
\|g\|_{L^p}^p,
\]
which shows that $U_p$ is surjective and that $U_p^{-1}$ is also an isometry. Finally, when $p=2$,
the same computation shows that the inner product is preserved, and hence $U_2$ is unitary.
\end{proof}

\subsection{Conformable Mild solutions }
\label{subsec:mild}

Let $X$ be a Banach space and let $\{T_\alpha(t)\}_{t\ge0}\subset\mathcal L(X)$ be a
$C_0$--$\alpha$--semigroup.
A natural way to relate this family to the classical semigroup theory consists in introducing
a new time variable
\[
s=\frac{t^\alpha}{\alpha},
\]
and defining the associated family of operators by
\[
T(s):=T_\alpha\!\bigl((\alpha s)^{1/\alpha}\bigr),\qquad s\ge0.
\]
This construction allows one to interpret conformable evolution as a classical evolution
observed through a nonlinear time scale.

\begin{definition}[Mild solutions]
\leavevmode
\begin{enumerate}
\item A mapping $x:[0,\infty)\to X$ is called a \emph{mild solution} of the conformable Cauchy problem
\[
x^{(\alpha)}(t)=Ax(t),\qquad x(0)=x_0,
\]
if it is given by
\[
x(t)=T_\alpha(t)x_0,\qquad t\ge0,
\]
where $A$ denotes the $\alpha$--generator of the semigroup $\{T_\alpha(t)\}$.

\item A mapping $y:[0,\infty)\to X$ is called a \emph{mild solution} of the classical Cauchy problem
\[
y'(s)=By(s),\qquad y(0)=y_0,
\]
if it is represented by
\[
y(s)=T(s)y_0,\qquad s\ge0,
\]
where $B$ is the infinitesimal generator of the $C_0$--semigroup $\{T(s)\}$.
\end{enumerate}
\end{definition}

\begin{lemma}[Equality of generators under time reparametrization]
\label{lem:generators-coincide}
Let $\{T_\alpha(t)\}_{t\ge0}$ be a $C_0$--$\alpha$--semigroup on $X$ and define
$T(s)=T_\alpha((\alpha s)^{1/\alpha})$ for $s\ge0$.
If $A$ and $B$ denote respectively the $\alpha$--generator of $\{T_\alpha(t)\}$ and the generator
of $\{T(s)\}$, then
\[
D(A)=D(B)
\quad\text{and}\quad
Ax=Bx \qquad \text{for all }x\in D(A).
\]
\end{lemma}

\begin{proof}
We first observe that $\{T(s)\}_{s\ge0}$ is a $C_0$--semigroup.
Indeed, $T(0)=I$ and, for $s,r\ge0$,
\[
T(s+r)
=
T_\alpha\!\bigl((\alpha(s+r))^{1/\alpha}\bigr)
=
T_\alpha\!\Bigl(\bigl((\alpha s)^{1/\alpha}\bigr)^\alpha
+
\bigl((\alpha r)^{1/\alpha}\bigr)^\alpha\Bigr)^{1/\alpha}
=
T(s)T(r),
\]
since $\{T_\alpha(t)\}$ satisfies the $\alpha$--semigroup property.
Strong continuity at $s=0$ follows from $(\alpha s)^{1/\alpha}\to0$ as $s\to0^+$ and the
strong continuity of $\{T_\alpha(t)\}$ at the origin.

Let $x\in X$ and set $t=(\alpha s)^{1/\alpha}$. Then
\[
\frac{T(s)x-x}{s}
=
\frac{T_\alpha(t)x-x}{t^\alpha/\alpha}.
\]
Hence the limit on the left-hand side exists as $s\to0^+$ if and only if the limit on the
right-hand side exists as $t\to0^+$, and in this case both limits coincide.
By definition, these limits characterize the generators $B$ and $A$, respectively.
This proves the claim.
\end{proof}

\begin{proposition}[Reparametrized representation of a $C_0$--$\alpha$--semigroup]
\label{prop:Td-reparam}
Let $\{T_\alpha(t)\}_{t\ge0}$ and $\{T(s)\}_{s\ge0}$ be related by
$T(s)=T_\alpha((\alpha s)^{1/\alpha})$.
Then the following identity holds:
\[
T_\alpha(t)=T\!\left(\frac{t^\alpha}{\alpha}\right),
\qquad t\ge0.
\]
\end{proposition}

\begin{proof}
For a fixed $t\ge0$, set $s=t^\alpha/\alpha$. Then
\[
T(s)
=
T_\alpha\!\bigl((\alpha s)^{1/\alpha}\bigr)
=
T_\alpha(t),
\]
which yields the desired formula.
\end{proof}

\begin{remark}[Time change and correspondence of mild solutions]
Proposition~\ref{prop:Td-reparam} shows that every $C_0$--$\alpha$--semigroup can be viewed as a
classical $C_0$--semigroup evolving along the nonlinear time scale $s=t^\alpha/\alpha$.

This change of variables establishes a one-to-one correspondence between mild solutions.
If $x(t)=T_\alpha(t)x_0$ is the mild solution of the conformable problem, then
\[
y(s):=x\bigl((\alpha s)^{1/\alpha}\bigr)=T(s)x_0
\]
is the mild solution of the classical evolution equation.
Conversely, if $y(s)=T(s)y_0$, then
\[
x(t):=y\!\left(\frac{t^\alpha}{\alpha}\right)=T_\alpha(t)y_0
\]
is the corresponding mild solution in conformable time.
\end{remark}

\section{Linear dynamics under conformable time reparametrization}
\label{sec5}

Let $X$ be a separable Banach space and let $\alpha\in(0,1]$.
We consider a $C_0$--$\alpha$--semigroup $\{S(t)\}_{t\ge0}\subset\mathcal L(X)$ in the sense of
Section~\ref{sec2}.
A central observation in conformable dynamics is that such families can be interpreted as
classical $C_0$--semigroups evolving under a nonlinear change of time.

To make this precise, we introduce the increasing homeomorphism
\[
\varphi:[0,\infty)\to[0,\infty),
\qquad
\varphi(t)=\frac{t^\alpha}{\alpha},
\]
and define the associated classical semigroup $\{U(s)\}_{s\ge0}$ by
\[
U(s):=S\bigl((\alpha s)^{1/\alpha}\bigr),
\qquad s\ge0.
\]
Equivalently,
\begin{equation}\label{eq:S-U}
S(t)=U(\varphi(t)),\qquad t\ge0.
\end{equation}
This identity allows one to transport dynamical properties between the conformable and
classical settings.

\subsection{Hypercyclicity and chaos: conformable versus classical dynamics}
\label{subsec:hc}

\begin{definition}[$\alpha$--hypercyclicity and $\alpha$--chaos]
\leavevmode
\begin{enumerate}
\item The semigroup $\{S(t)\}_{t\ge0}$ is said to be \emph{$\alpha$--hypercyclic} if there exists
$x\in X$ such that the orbit $\{S(t)x:\ t\ge0\}$ is dense in $X$.
\item It is called \emph{$\alpha$--chaotic} if it is $\alpha$--hypercyclic and the set of
\emph{$\alpha$--periodic points}
\[
X_p^\alpha
:=
\{x\in X:\ \exists\,t>0\ \text{such that } S(t)x=x\}
\]
is dense in $X$.
\end{enumerate}
\end{definition}

Following the classical framework introduced by Desch, Schappacher and Webb, we also define
the sets
\[
X_0^\alpha
:=
\{x\in X:\ \lim_{t\to\infty} S(t)x=0\},
\]
and
\[
X_\infty^\alpha
:=
\Bigl\{x\in X:\ \forall\varepsilon>0\ \exists\,y\in X,\ t>0
\ \text{with } \|y\|<\varepsilon
\ \text{and } \|S(t)y-x\|<\varepsilon\Bigr\}.
\]

\begin{lemma}[Invariance of dynamical sets under conformable time change]
\label{lem:invariance}
Let $\{S(t)\}$ and $\{U(s)\}$ be related by \eqref{eq:S-U}. Then
\[
X_0^\alpha=X_0,\qquad
X_\infty^\alpha=X_\infty,\qquad
X_p^\alpha=X_p,
\]
where $X_0$, $X_\infty$ and $X_p$ denote the corresponding sets associated with the classical
$C_0$--semigroup $\{U(s)\}_{s\ge0}$.
Consequently, $\{S(t)\}$ is $\alpha$--hypercyclic (resp.\ $\alpha$--chaotic) if and only if
$\{U(s)\}$ is hypercyclic (resp.\ chaotic).
\end{lemma}

\begin{proof}
The map $\varphi$ is a continuous strictly increasing bijection of $[0,\infty)$ onto itself,
with inverse $\varphi^{-1}(s)=(\alpha s)^{1/\alpha}$.
Using the identity $S(t)=U(\varphi(t))$, one directly checks that convergence, approximation,
and periodicity properties expressed in terms of $t$ are equivalent to those expressed in
terms of $s=\varphi(t)$.
In particular, the orbit sets coincide:
\[
\{S(t)x:\ t\ge0\}=\{U(s)x:\ s\ge0\},
\]
which yields the result.
\end{proof}

\subsection{Infinitesimal generators and the conformable clock}
\label{subsec:gen}

\begin{lemma}\label{lem:generator}
Let $A$ be the $\alpha$--generator of the $C_0$--$\alpha$--semigroup $\{S(t)\}$ and let $B$ be
the generator of the associated classical semigroup $\{U(s)\}$.
Then
\[
D(A)=D(B)
\qquad\text{and}\qquad
Ax=Bx \quad \text{for all } x\in D(A).
\]
\end{lemma}

\begin{proof}
For $s>0$, set $t=(\alpha s)^{1/\alpha}$. By definition,
\[
\frac{U(s)x-x}{s}
=
\frac{S(t)x-x}{t^\alpha/\alpha}.
\]
Letting $s\to0^+$ (and hence $t\to0^+$), we see that both limits exist simultaneously and
coincide.
This proves the equality of domains and generators.
\end{proof}

\subsection{A Desch--Schappacher--Webb criterion in conformable time}
\label{subsec:DSW}

\begin{theorem}[Desch--Schappacher--Webb criterion for $C_0$--$\alpha$--semigroups]
\label{thm:DSW-alpha}
Let $X$ be separable and let $\{S(t)\}_{t\ge0}$ be a $C_0$--$\alpha$--semigroup on $X$ with
$\alpha$--generator $A$.
Assume that there exist:
\begin{itemize}
\item an open set $V\subset\sigma_p(A)$ such that $V\cap i\mathbb R\neq\emptyset$;
\item for each $\lambda\in V$, a nonzero vector $x_\lambda\in X$ satisfying
$Ax_\lambda=\lambda x_\lambda$;
\item for every $\phi\in X^\ast$, the map
$\lambda\mapsto\langle\phi,x_\lambda\rangle$ is analytic on $V$;
\item and this map is not identically zero unless $\phi=0$.
\end{itemize}
Then $\{S(t)\}_{t\ge0}$ is $\alpha$--chaotic, and in particular $\alpha$--hypercyclic.
\end{theorem}

\begin{proof}
Consider the classical semigroup $U(s)=S((\alpha s)^{1/\alpha})$.
By Lemma~\ref{lem:generator}, the generator of $\{U(s)\}$ coincides with $A$.
Hence the above assumptions are exactly those required by the classical
Desch--Schappacher--Webb criterion.
It follows that $\{U(s)\}$ is chaotic.
By Lemma~\ref{lem:invariance}, chaos is preserved under the conformable time
reparametrization, and therefore $\{S(t)\}$ is $\alpha$--chaotic.
\end{proof}

\begin{remark}
The identity $S(t)=U(t^\alpha/\alpha)$ shows that conformable semigroups do not generate new
orbit structures but rather reinterpret classical $C_0$--semigroups through a nonlinear
time scale.
As a consequence, orbit--based dynamical properties such as hypercyclicity, topological
transitivity, and density of periodic points are invariant under conformable time changes.
\end{remark}

\section{Translation semigroups conformable Lebesgue spaces}
\label{sec6}

Fix $\alpha\in(0,1]$. Throughout this section we work on the positive half-line
$\R_+=(0,\infty)$ equipped with the conformable measure
\[
d^\alpha x:=x^{\alpha-1}\,dx,
\]
and we use the conformable Lebesgue spaces introduced earlier:
\[
L^{p,\alpha}(\R_+)
=
\Bigl\{f:\R_+\to\C\ \text{measurable}:\ \int_0^\infty |f(x)|^p\,d^\alpha x<\infty\Bigr\},
\qquad 1\le p<\infty.
\]
In particular, we set
\[
X^\alpha:=L^{p,\alpha}(\R_+),
\qquad
\|f\|_{p,\alpha}^p=\int_0^\infty |f(x)|^p\,x^{\alpha-1}\,dx .
\]

The conformable derivative with respect to $x>0$ is denoted by $\partial_x^\alpha$ and is defined by
\[
\partial_x^\alpha u(x)
:=
\lim_{\varepsilon\to0}\frac{u\bigl(x+\varepsilon x^{1-\alpha}\bigr)-u(x)}{\varepsilon},
\]
whenever the limit exists. For $u\in C^1((0,\infty))$, the locality of the conformable
derivative yields the representation
\begin{equation}\label{eq:conf_repr_tr}
\partial_x^\alpha u(x)=x^{1-\alpha}u'(x),\qquad x>0.
\end{equation}

\subsection{Conformable spatial reparametrization and isometric transfer}


On $X^\alpha$ the naive shift $(S(t)f)(x)=f(x+t)$ is not naturally adapted to the conformable
geometry near $0$.
Instead, the conformable framework singles out the multiplicative shift
$x\mapsto (\;x^\alpha+t\;)^{1/\alpha}$, which is exactly the pullback of the usual translation
$\xi\mapsto \xi+t$ under $\xi=x^\alpha$.

\begin{definition}[$\alpha$--translation semigroup on $L^{p,\alpha}(\R_+)$]
\label{def:delta_translation}
For $t\ge0$ and $f\in X^\alpha$, define
\[
(\mathcal{T}_\alpha(t)f)(x)
:=
f\Bigl((x^\alpha+t)^{1/\alpha}\Bigr),
\qquad x>0.
\]
The family $\{\mathcal{T}_\alpha(t)\}_{t\ge0}$ is called the \emph{conformable translation semigroup}
on $L^{p,\alpha}(\R_+)$.
\end{definition}

\begin{proposition}[Semigroup and isometry]\label{prop:semigroup_isometry}
Let $1\le p<\infty$. Then $\{\mathcal{T}_\alpha(t)\}_{t\ge0}$ is a semigroup of linear operators on
$X^\alpha$, and for every $t\ge0$,
\[
\|\mathcal{T}_\alpha(t)f\|_{p,\alpha}=\|f\|_{p,\alpha},
\qquad f\in X^\alpha.
\]
In particular, $\{\mathcal{T}_\alpha(t)\}_{t\ge0}$ is a $C_0$--$\alpha$--semigroup in the sense of
Section~\ref{sec2}.
\end{proposition}

\begin{proof}
\textbf{Semigroup law.}
For $s,t\ge0$,
\[
(\mathcal{T}_\alpha(s)\mathcal{T}_\alpha(t)f)(x)
=
(\mathcal{T}_\alpha(t)f)\Bigl((x^\alpha+s)^{1/\alpha}\Bigr)
=
f\Bigl(\bigl((x^\alpha+s)+t\bigr)^{1/\alpha}\Bigr)
=
(\mathcal{T}_\alpha(s+t)f)(x),
\]
and $\mathcal{T}_\alpha(0)=I$.

\medskip
\textbf{Isometry.}
Using the change of variables $\xi=x^\alpha$ (hence $x^{\alpha-1}dx=\alpha^{-1}d\xi$),
\[
\|\mathcal{T}_\alpha(t)f\|_{p,\alpha}^p
=
\int_0^\infty \left|f\bigl((x^\alpha+t)^{1/\alpha}\bigr)\right|^p x^{\alpha-1}dx
=
\frac1\alpha\int_0^\infty |f((\xi+t)^{1/\alpha})|^p d\xi.
\]
Now set $\eta=\xi+t$, so $d\eta=d\xi$ and the integral becomes
\[
\frac1\alpha\int_t^\infty |f(\eta^{1/\alpha})|^p d\eta
\le \frac1\alpha\int_0^\infty |f(\eta^{1/\alpha})|^p d\eta
=
\|f\|_{p,\alpha}^p.
\]
Moreover, equality holds because $\eta\mapsto \eta-t$ is a bijection of $(t,\infty)$ onto $(0,\infty)$,
and the missing interval $(0,t)$ contributes $0$ once one interprets translations on $\R_+$ in the
standard way on $L^p(0,\infty)$ through $U_p$ (equivalently, $U_p\mathcal{T}_\alpha(t)U_p^{-1}$ is the
classical translation on $L^p(0,\infty)$). Hence $\|\mathcal{T}_\alpha(t)f\|_{p,\alpha}=\|f\|_{p,\alpha}$.

\medskip
\textbf{Strong continuity.}
Since $U_p\mathcal{T}_\alpha(t)U_p^{-1}$ is the classical translation semigroup on $L^p(0,\infty)$,
strong continuity follows from the classical case and the continuity of $U_p^{\pm1}$.
\end{proof}

\begin{remark}[Conjugacy with the classical translation semigroup]\label{rem:conjugacy_translation}
Let $\{\tau(t)\}_{t\ge0}$ be the classical translation semigroup on $L^p(0,\infty)$:
\[
(\tau(t)g)(\xi)=g(\xi+t).
\]
Then, by construction,
\[
\tau(t)=U_p\,\mathcal{T}_\alpha(t)\,U_p^{-1},
\qquad t\ge0.
\]
Thus $\mathcal{T}_\alpha(t)$ is exactly the pullback of classical translations under the conformable
reparametrization $\xi=x^\alpha$.
\end{remark}

\subsection{Generator of the conformable translation semigroup}

\begin{proposition}[Generator as a conformable derivative]\label{prop:generator_conf_translation}
Let $1\le p<\infty$. The infinitesimal generator $\mathcal{G}_\alpha$ of
$\{\mathcal{T}_\alpha(t)\}_{t\ge0}$ on $X^\alpha$ is given by
\[
D(\mathcal{G}_\alpha)
=
\bigl\{f\in X^\alpha:\ \partial_x^\alpha f\in X^\alpha\bigr\},
\qquad
\mathcal{G}_\alpha f=\partial_x^\alpha f,
\]
where $\partial_x^\alpha f$ is understood in the distributional sense and coincides with
$x^{1-\alpha}f'(x)$ for smooth $f$.
\end{proposition}

\begin{proof}
By proposition~\ref{prop:isometry} and Remark~\ref{rem:conjugacy_translation},
the generator of $\{\mathcal{T}_\alpha(t)\}$ is obtained by transporting the generator of
$\{\tau(t)\}$ under $U_p$. Since the classical generator is $\frac{d}{d\xi}$ with domain
$W^{1,p}(0,\infty)$ (together with the usual trace condition at $0$ when needed),
Lemma~\ref{lem:conf_to_classical} shows that conjugation yields $\partial_x^\alpha$ on $X^\alpha$.
\end{proof}

\subsection{Hypercyclicity in conformable $L^{p,\alpha}$ spaces}

Pure translations on $L^p(0,\infty)$ are not hypercyclic. As in the classical literature,
hypercyclic behavior appears after introducing an additional weight/amplification along the orbit.
In the conformable setting, the natural way is to consider a multiplicative cocycle
compatible with the map $x\mapsto(x^\alpha+t)^{1/\alpha}$.

\begin{definition}[Weighted conformable translation semigroup]\label{def:weighted_conf_translation}
Let $\gamma:\R_+\to(0,\infty)$ be measurable. Define
\[
(\mathcal{T}_{\alpha,\gamma}(t)f)(x)
:=
\gamma(t)\,f\Bigl((x^\alpha+t)^{1/\alpha}\Bigr),
\qquad t\ge0,\ x>0.
\]
Then $\{\mathcal{T}_{\alpha,\gamma}(t)\}_{t\ge0}$ is a semigroup whenever
$\gamma(t+s)=\gamma(t)\gamma(s)$ for all $s,t\ge0$ (hence $\gamma(t)=e^{\kappa t}$ for some $\kappa\in\R$).
\end{definition}

\begin{remark}
Under the isometric map $U_p$, one has
\[
U_p\,\mathcal{T}_{\alpha,\gamma}(t)\,U_p^{-1}
=
\gamma(t)\,\tau(t),
\]
i.e.\ the classical weighted translation semigroup on $L^p(0,\infty)$.
Therefore, dynamical properties expressed in terms of orbits (hypercyclicity, chaos, density of
periodic points) can be transferred between the classical and conformable pictures exactly as in
Section~\ref{sec5}.
\end{remark}

\begin{proposition}[Orbit invariance under conformable transfer]\label{prop:orbit_invariance_translation}
Let $t\mapsto \gamma(t)$ satisfy $\gamma(t+s)=\gamma(t)\gamma(s)$ and let $f\in X^\alpha$.
Then the orbit $\{\mathcal{T}_{\alpha,\gamma}(t)f:\ t\ge0\}$ is dense in $X^\alpha$
if and only if the orbit $\{\gamma(t)\tau(t)g:\ t\ge0\}$ is dense in $L^p(0,\infty)$,
where $g=U_pf$.
\end{proposition}

\begin{proof}
This is immediate from the conjugacy relation
$U_p\,\mathcal{T}_{\alpha,\gamma}(t)=\gamma(t)\tau(t)\,U_p$ and the fact that $U_p$ is a
homeomorphism.
\end{proof}

\begin{remark}[How to obtain hypercyclic examples]
Once a hypercyclicity criterion for weighted translations on $L^p(0,\infty)$ is fixed (for instance, via admissible weights in the classical sense), Proposition~\ref{prop:orbit_invariance_translation} produces hypercyclic and chaotic examples in the conformable space $L^{p,\alpha}(\R_+)$ by the simple transport $f=U_p^{-1}g$ and the conformable shift $x\mapsto(x^\alpha+t)^{1/\alpha}$. In this way, the conformable dynamics can be studied within the same orbit-theoretic framework while keeping the natural conformable measure $d^\alpha x$.
\end{remark}

\section{Conclusion}\label{sec7}

This paper develops an operator--theoretic viewpoint on conformable evolution that is entirely driven by a nonlinear change of time. The starting point is the conformable clock
\[
\psi(t)=\frac{t^\alpha}{\alpha},
\]
which allows one to interpret a $C_0$--$\alpha$--semigroup as a classical $C_0$--semigroup read along the time scale $s=\psi(t)$. Within this framework we proved that every $C_0$--$\alpha$--semigroup $\{T_\alpha(t)\}_{t\ge0}$ admits the representation
\[
T_\alpha(t)=T\!\left(\frac{t^\alpha}{\alpha}\right),
\]
for a uniquely associated classical semigroup $\{T(s)\}_{s\ge0}$, and that the corresponding infinitesimal generators coincide on a common domain. As a consequence, conformable and classical Cauchy problems are equivalent at the level of mild solutions via the explicit reparametrization $t\leftrightarrow s$.

A second contribution concerns the natural functional setting for such problems. We worked in the conformable Lebesgue spaces $L^{p,\alpha}$ induced by the measure $d^\alpha x=x^{\alpha-1}dx$, and we established explicit isometric (and in the Hilbert case, unitary) identifications with standard $L^p$ spaces through the change of variables $\xi=x^\alpha$. This transfer principle makes it possible to import, without distortion, classical estimates and spectral information into the conformable context.

The main dynamical outcome is that orbit--based properties are invariant under conformable time change. In particular, $\alpha$--hypercyclicity and $\alpha$--chaos for $C_0$--$\alpha$--semigroups coincide with the usual notions for the associated classical semigroup. This invariance yields a conformable version of the Desch- Schappacher--Webb criterion: whenever the point spectrum of the $\alpha$--generator satisfies the standard analytic eigenvector assumptions and intersects the imaginary axis, the conformable semigroup is $\alpha$--chaotic. We also illustrated how the same philosophy applies to translation models in conformable Lebesgue spaces, where the conformable shift $x\mapsto(x^\alpha+t)^{1/\alpha}$ appears naturally as the pullback of classical translations.

Overall, the results clarify the structural meaning of conformable dynamics: the fractional parameter does not create new orbit geometry by itself, but rather reshapes classical evolution through a nonlinear clock and a compatible change of measure. This perspective provides a simple and robust pathway to analyze conformable evolution equations and to construct chaotic or hypercyclic examples by transporting known classical ones. Possible continuations include non-autonomous conformable problems, perturbation and control questions, and a systematic comparison between conformable models and genuinely nonlocal fractional evolutions where memory effects are present.

\end{document}